\documentclass{amsart}

\usepackage[a4paper,hmargin=3.5cm,vmargin=4cm]{geometry}
\usepackage{amsfonts,amssymb,amscd,amstext}
\usepackage{graphicx}
\usepackage[dvips]{epsfig}
\usepackage{mathpazo}

\renewcommand{\leq}{\leqslant}
\renewcommand{\geq}{\geqslant}

\newtheorem{theorem}{Theorem}[section]
\newtheorem{proposition}[theorem]{Proposition}
\newtheorem{lemma}[theorem]{Lemma}

\theoremstyle{definition}
\newtheorem{definition}[theorem]{Definition}
\newtheorem{remark}[theorem]{Remark}

\theoremstyle{remark}

\numberwithin{equation}{section}

\newcommand{\dist}{\operatorname{dist}}
\newcommand{\Vol}{\operatorname{Vol}}

\newcommand{\kan}{\mathbb{K}^{n}(b)}
\newcommand{\kam}{\mathbb{K}^{m}(b)}
\newcommand{\han}{\mathbb{H}^{n}(b)}
\newcommand{\ham}{\mathbb{H}^{m}(b)}
\newcommand{\erre}{\mathbb{R}}

\numberwithin{equation}{section}

\begin{document}

\title[Exit Moment Spectra]
{Comparison of exit moment spectra \\ for extrinsic metric balls}

\author[Ana Hurtado]{Ana Hurtado$^{\natural}$}
\address{Departamento de Geometr\'{\i}a y Topolog\'{\i}a, Universidad de Granada, E-18071,
Spain.}
 \email{ahurtado@ugr.es}
\author[Steen Markvorsen]{Steen Markvorsen$^{\#}$}
\address{Department of Mathematics, Technical University of Denmark,  DK-2800 Kgs. Lyngby, Denmark}
\email{S.Markvorsen@mat.dtu.dk}
\author[Vicente Palmer]{Vicente Palmer*}
\address{Departament de Matem\`{a}tiques, Universitat Jaume I, Castell\'o,
Spain.} \email{palmer@mat.uji.es}
\thanks{$^{\#}$ Supported by the Danish Natural Science Research Council and the Spanish MEC-DGI grant
MTM2007-62344.\\
\indent * Supported by Spanish Micinn-DGI grant No.MTM2007-62344 and
the Caixa Castell\'o Foundation\\
\indent $^{\natural}$ Supported by Spanish Micinn-DGI grant
No.MTM2007-62344 and the Caixa Castell\'o Foundation}

%


\begin{abstract}
We prove explicit upper and lower bounds for the $L^1$-moment
spectra for the Brownian motion exit time from extrinsic metric balls of  submanifolds  $P^m$ in ambient Riemannian spaces $N^{n}$.
We assume that $P$ and $N$ both have controlled radial curvatures (mean curvature and sectional curvature, respectively) as viewed from a pole in $N$.
The  bounds for the  exit moment spectra are given in
terms of the corresponding spectra for geodesic metric balls in suitably
warped product model spaces. The bounds are sharp in the sense that
equalities are obtained in characteristic cases. As a corollary we also obtain
new intrinsic comparison results for the exit time
spectra for metric balls in the ambient manifolds $N^n$ themselves.
\end{abstract}

\maketitle

\section{Introduction}\label{secIntro}
\bigskip
We consider a complete Riemannian manifold $(M^n, g)$ and the
induced Brownian motion $X_t$ defined on $M$. The $L^{p}$-moments of the exit
time of $X_t$ from smooth precompact domains $D$ in the manifold
are given by the following integrals (see  \cite{H, KD, KDM, Mc, Dy}):
\begin{equation}\label{defmoment}
\mathcal{A}_{p,\,k}(D)=  \left(\int_D \left(u_k(x)\right)^{p}\, dV\right)^{1/p} \quad ,
\end{equation}
where the functions $u_k$ are defined inductively as the sequence of
solutions to the following hierarchy of boundary value problems

\begin{equation}
\begin{aligned}
\Delta u_1+1 &= 0\,\,\, \text{on}\,\,\, D\\
u_1\vert_{\partial D} &=0 \quad,
\end{aligned}
\end{equation}
and for $k\geq 2$,
\begin{equation}  \label{eqmoments1}
\begin{aligned}
\Delta u_k+k\, u_{k-1} &= 0\,\,\, \text{on}\,\,\, D\\
u_k\vert_{\partial D} &=0 \quad .
\end{aligned}
\end{equation}

Here  $\Delta$ denotes the Laplace-Beltrami operator on $\,(M^{n},
g)\,$. The first solution $u_1(x)$ is the mean time of first
exit from $D$ for the Brownian motion starting at the point $x$ in
$D$, see \cite{Dy, Ma1}.

The quantity $\mathcal{A}_{1,1}(D)$ is known as the \emph{torsional rigidity}
of $D$. This name stems from the fact that  if $D \subseteq
\erre^2$, then $\mathcal{A}_{1,1}(D)$ represents the torque required
per unit angle of twist and per unit beam length
    when twisting an elastic beam of
    uniform cross section $D$, see \cite{Ba} and  \cite{PS}. The
    torsional rigidity plays a role in the exit moment spectrum
    similar to the role played by the first positive Dirichlet
    eigenvalue in the Dirichlet spectrum. See also \cite{Ch1, Ch2} and  \cite{BBC, BG}.

    Perhaps the most relevant example and token of interest in these problems is given
by the St. Venant torsion problem. It is a precise analog of the
Rayleigh conjecture about the fundamental tone of a membrane. In
1856 Saint-Venant conjectured that among all cross sections with a
given area, the circular disk has maximum torsional rigidity. The first
proof of this conjecture was given by G. P´olya in 1948, see \cite{Po} and \cite{PS}.

In view of the isoperimetric inequality for domains in $\mathbb{R}^{2}$ and in view of the domain monotonicity of $\mathcal{A}_{1,1}(D)$ it thence follows, that
among all cross sections with a given \emph{circumference}, the circular disk has maximum
torsional rigidity. In other words, in $\mathbb{R}^{2}$ the boundary-relative torsional rigidity is maximized by the circular
disks.

Since we shall similarly only be concerned with $p=1$, and since our results for the higher moments in the exit time moment spectrum are
also in this sense isoperimetric type inequalities we
define:
\begin{definition} \label{defAhat}
The \emph{isoperimetric exit moment spectrum of $D$} is defined by $\{\widehat{\mathcal{A}}_{1}(D), \widehat{\mathcal{A}}_{2}(D), \cdots   \}$, where
\begin{equation} \label{eqDefAhat}
\widehat{\mathcal{A}}_k(D) = \frac{\mathcal{A}_{1,k}(D)}{\Vol(\partial D)} = \frac{1}{\Vol(\partial D)}\int_D \,u_k(x)\, dV   \quad .
\end{equation}
\end{definition}

If we formally define $u_{0}(x) = 1$ for all $x \in D$, then all the solutions $u_{k}$ -- including $u_{1}(x)$ -- are uniformly generated by induction from (\ref{eqmoments1}). With this  natural extension of the $u_{k}$ sequence we thence have from Definition \ref{defAhat}:
\begin{equation}
\widehat{\mathcal{A}}_{0}(D)  = \frac{1}{\Vol(\partial D)}\int_D \,u_0(x)\, dV   = \frac{\Vol(D)}{\Vol(\partial D)}\quad ,
\end{equation}
which is precisely the isoperimetric quotient for $D$.

We will henceforth refer to the list $\{\widehat{\mathcal{A}}_{0}(D), \widehat{\mathcal{A}}_{1}(D), \widehat{\mathcal{A}}_{2}(D), \cdots   \}$ as  the \emph{extended} isoperimetric exit moment spectrum of $D$.

Here we restrict our study to be concerned with the exit moment spectra of a specific kind of domains, the so-called extrinsic
$R$-balls $D_R$ defined in  submanifolds $P^m$ which are properly immersed into ambient Riemannian manifolds $N^{n}$ with controlled sectional curvatures.

Suppose $p$ is a pole in $N$, see \cite{S}. An extrinsic $p$-centered $R$-ball $D_R$ of the submanifold $P$ is then, roughly speaking, the intersection between the submanifold and the ambient metric $R$-ball centered at $p$ in the ambient space $N$.

The isoperimetric relations satisfied by these extrinsic balls have been studied and applied in a number of contexts, see e.g. \cite{Pa2, MP1, MP4, HMP, MP5}. In these works we use $R$-balls and $R$-spheres in tailor made rotationally symmetric (warped product) model spaces $M^{m}_{w}$ as comparison objects.

The simplest settings considered are given by the minimal submanifolds $P^m$ in  real space forms $\kan$ of constant sectional curvature $b \leq 0$. In these specific cases we have the following isoperimetric inequalities, see \cite{CLY, Ma1, Ma2, Pa2, MP1}:

\begin{equation}
\frac{\Vol(D_R)}{\Vol(\partial D_R)} \leq \frac{\Vol(B^{b,m}_R)}{\Vol(S^{b,m-1}_R)} \quad ,
\end{equation}
where $B^{b,m}_R$ and $S^{b,m-1}_R = \partial B^{b,m}_R$ denote, respectively, the geodesic $R$-ball and the geodesic $R$-sphere in the real space form $\kan$.

With the notation introduced above we may state this result as follows:
\begin{equation} \label{eqA0}
\widehat{\mathcal{A}}_{0}(D_{R}) \leq \widehat{\mathcal{A}}_{0}(B_{R}^{b, m})  \quad .
\end{equation}

In passing we note that when equality is attained in (\ref{eqA0}) for some fixed radius $R$, and when the ambient space $N^{n}$ is the hyperbolic space $\han$, $b < 0$, then the minimal submanifold itself is a totally geodesic hyperbolic subspace $\ham$ of $\han$, see \cite{Pa2}. Thus, in analogy with the  St. Venant torsion problem -- and in analogy with the classical isoperimetric problem itself -- we also obtain strong rigidity conclusions from equalities in these isoperimetric estimates.

\subsection{A first glimpse of the main results}

In the present paper we extend the inequalities (\ref{eqA0}) and prove isoperimetric inequalities of this type for \emph{every} element $\widehat{\mathcal{A}}_{k}(D_{R})$, $k \geq 0$, in the extended isoperimetric exit moment spectrum for extrinsic metric balls.

Before stating this extension for
\emph{minimal submanifolds} in constant curvature ambient spaces below we note, that this is but a shadow of our main results, Theorem \ref{thm2.1} and Theorem \ref{thm2.2} in section \ref{MainSect}, where we prove both upper \emph{and lower} bounds for the isoperimetric exit moment spectrum under more relaxed curvature conditions. The main condition for the lower bounds is a lower bound on the sectional curvatures of the ambient space and the
upper bounds for the spectrum stem similarly from an upper bound on the ambient sectional curvatures. Moreover, in our \emph{general results} the submanifolds are not assumed beforehand to be minimal.

\begin{theorem} \label{thmExFirst}
Let $P^m$ be a minimal submanifold properly immersed in the real space form $\kan$ with constant sectional curvature $b\leq 0$.
 Let $D_R$ be an
extrinsic $R$-ball in $P^m$,
 with center at a point $p \in P$.  Then we have for the extended isoperimetric exit moment spectrum of $D_{R}$, i.e. for all $k\geq 0$:
\begin{equation} \label{eqExFirst}
\widehat{\mathcal{A}}_{k}(D_R)\leq \widehat{\mathcal{A}}_{k}(B^{b,m}_{R})  \quad ,
\end{equation}
where $B^{b,m}_{R}$ is the geodesic ball of radius $R$ in $\kam$.

When the ambient space is hyperbolic space $\mathbb{H}^{n}(b)$, $b <0$,  then equality in (\ref{eqExFirst}) for some radius $R$ and for some  value of $k \geq 0$ implies that $D_{R}$ -- and in fact all of $P^{m}$ -- is totally geodesic in $\mathbb{H}^{n}(b)$, so that equality is attained for all $k$ and for every smaller $p$-centered extrinsic ball in $P^{m}$.
\end{theorem}

In order to illustrate our use of the upper and lower bounds on the ambient space sectional curvatures in the more general setting alluded to above -- and since we believe that the following result is also in itself of independent interest -- we extract here a purely \emph{intrinsic} consequence from the proofs of Theorems \ref{thm2.1} and \ref{thm2.2}. The notion of radial sectional curvatures and the geometric analytic notions associated with the model spaces are defined precisely in section \ref{pre} below.

\begin{theorem} \label{thm2.1intrinsic}
Let $B^N_R$ be a geodesic ball of a complete Riemannian manifold
$N^n$ with a pole $p$ and suppose that the $p$-radial sectional
curvatures of $N^n$  are
bounded from below (respectively from above) by the $p_w$-radial
sectional curvatures of a $w$-warped model space $M^n_w$.  Then the extended isoperimetric exit moment spectrum of $B^N_R$ satisfies for all $k\geq 0$ the following respective inequalities:
\begin{equation} \label{eqMain6}
\widehat{\mathcal{A}}_k(B^N_R)\geq (\leq)
\widehat{\mathcal{A}}_k(B^{w}_R) \quad ,
\end{equation}
where $B^{w}_{R}$ is the geodesic ball in the model space $M^n_w$.

Equality in (\ref{eqMain6}) for some $k \geq 0$ implies that $B^N_R$ is isometric to the warped product model ball $B^{w}_R$ and hence again that equality is attained for all $k \geq 0$ and for every smaller $p$-centered extrinsic ball in $P^{m}$.
\end{theorem}

The proofs of these results, Theorem \ref{thmExFirst} and \ref{thm2.1intrinsic} are given in section \ref{secIntrinsic} at the end of this paper.

\section{Preliminaries and Comparison Setting}\label{pre}

We first consider a few conditions and concepts that will
be instrumental for establishing our results.

\subsection{Extrinsic metric balls}

We consider a properly immersed $m$-dimensional submanifold $P^m$
in a complete Riemannian manifold $N^n$. Let $p$ denote a point in
$P$ and assume that $p$ is a pole of the ambient manifold $N$. We
denote the distance function from $p$ in $N^{n}$ by $r(x) =
\dist_{N}(p, x)$ for all $x \in N$. Since $p$ is a pole there is -
by definition - a unique geodesic from $x$ to $p$ which realizes
the distance $r(x)$. We also denote by $r$ the restriction
$r\vert_P: P\longrightarrow \erre_{+} \cup \{0\}$. This
restriction is then called the extrinsic distance function from
$p$ in $P^m$. The corresponding extrinsic metric balls of
(sufficiently large) radius $R$ and center $p$ are denoted by
$D_R(p) \subseteq P$ and defined as any connected component which
contains $p$ of the set:
$$D_{R}(p) = B_{R}(p) \cap P =\{x\in P \,|\, r(x)< R\} \quad ,$$
where $B_{R}(p)$ denotes the geodesic $R$-ball around the pole $p$ in $N^n$. The
 extrinsic ball $D_R(p)$ is a connected domain in $P^m$, with
boundary $\partial D_{R}(p)$. Since $P^{m}$ is assumed to be unbounded and properly immersed into $N$, we
have for every  $R$ that $B_{R}(p) \cap P \neq P$.

\subsection{The curvature bounds} \label{subsecurvature}

We now present the curvature restrictions which constitute the geometric framework of our investigations.

\begin{definition}
Let $p$ be a point in a Riemannian manifold $M$
and let $x \in M-\{ p \}$. The sectional
curvature $K_{M}(\sigma_{x})$ of the two-plane
$\sigma_{x} \in T_{x}M$ is then called a
\textit{$p$-radial sectional curvature} of $M$ at
$x$ if $\sigma_{x}$ contains the tangent vector
to a minimal geodesic from $p$ to $x$. We denote
these curvatures by $K_{p, M}(\sigma_{x})$.
\end{definition}

In order to control the mean curvatures $H_P(x)$ of $P^{m}$ at distance $r$ from
$p$ in $N^{n}$ we introduce the following definition:

\begin{definition} The $p$-radial mean curvature function for $P$ in $N$
is defined in terms of the inner product of $H_{P}$ with the $N$-gradient of the
distance function $r(x)$ as follows:
$$
\mathcal{C}(x) = -\langle \nabla r(x), H_{P}(x) \rangle  \quad
{\textrm{for all}}\quad x \in P \,\, .
$$
\end{definition}

In the following definition, we are going to generalize the notion of {\em radial mean convexity condition}
introduced in \cite{MP5}, \cite{HMP}.

\begin{definition} (see \cite{MP5})
We say that the submanifold $P$ satisfies a {\em{radial mean convexity condition
from below}} controlled by a smooth radial function $h_{1}(r)$ (respectively, {\em{from above}} controlled by a smooth radial function $h_{2}(r)$) from the point $p \in P$ such that
\begin{equation}\label{eqH}
\begin{aligned}
\mathcal{C}(x)& \geq h_{1}(r(x))\,  {\textrm{for
all}}\,\,  x \in P\,\,\quad {\textrm{($h_{1}(r)$ bounds {\it from below}})}\\
\mathcal{C}(x)& \leq h_{2}(r(x))\,  {\textrm{for
all}}\,\,  x \in P\,\,\quad {\textrm{($h_{2}(r)$ bounds {\it from above}})}
\end{aligned}
\end{equation}
\end{definition}

The radial bounding functions $h_{1}(r)$ and $h_{2}(r)$ are related to the global
extrinsic geometry of the submanifold. For example, it is obvious
that  minimal submanifolds satisfy a radial mean convexity
condition from above and from below, with bounding functions $h_{2}=0$ and $h_{1}=0$.
On the other hand, it can be proved, see the works
\cite{Sp,DCW,Pa1,MP5}, that when the submanifold is a convex
hypersurface, then the constant function $h_{1}(r)=0$ is
a radial bounding function from below.\\

The final notion needed to describe our comparison setting is the
idea of {\it radial tangency}. If we denote by $\nabla r$ and
$\nabla^P r$ the
 gradients of
$r$ in $N$ and $P$ respectively, then
we have the following basic relation:
\begin{equation}\label{eq2.1}
\nabla r = \nabla^P r +(\nabla r)^\bot \quad ,
\end{equation}
where $(\nabla r)^\bot(q)$ is perpendicular to $T_qP$ for all $q\in P$.\\

When the submanifold $P$ is totally geodesic, then $\nabla
r=\nabla^P r$ in all points, and, hence, $\Vert \nabla^P r\Vert
=1$. On the other hand, and given the starting point $p \in P$,
from which we are measuring the distance $r$, we know that $\nabla
r(p)=\nabla^P r(p)$, so $\Vert \nabla^P r(p)\Vert =1$. Therefore,
the difference  $1 - \Vert \nabla^P r\Vert$ quantifies the radial
{\em detour} of the submanifold with respect the ambient manifold
as seen from the pole $p$. To control this detour locally, we
apply the following

\begin{definition}

 We say that the submanifold $P$ satisfies a {\it radial tangency
 condition} at $p\in P$ when we have a smooth positive function
$g(r)$ so that
\begin{equation}
\mathcal{T}(x) \, = \, \Vert \nabla^P r(x)\Vert
\geq g(r(x)) \, > \, 0  \quad {\textrm{for all}}
\quad x \in P \,\, .
\end{equation}
\end{definition}

\begin{remark}
Of course, we always have
\begin{equation}
\mathcal{T}(x) \, =\, \Vert \nabla^P r(x)\Vert
\leq1 \quad {\textrm{for all}}
\quad x \in P \,\, .
\end{equation}
\end{remark}

\begin{remark} \label{remTrivTop}
We observe, that the assumption  $ \Vert \nabla^P r(x)\Vert > 0 $ implies that
the properly immersed extrinsic ball $D_{R}$ in $P$ can have only trivial topology.
It follows directly from Theorem 3.1 in \cite{Mi}, since $r(x)$ is a smooth function on $P - \{p\}$ without critical points,  that $D_{R}$ is diffeomorphic to the
standard unit ball in $\mathbf{R}^{m}$.
\end{remark}

\subsection{Model Spaces} \label{secModel}

As mentioned previously, the model spaces $M^m_w$ serve first and foremost
as com\-pa\-ri\-son controller objects for the radial sectional
curvatures of $N^{n}$.

\begin{definition}[See \cite{Gri}, \cite{GreW}]
 A $w-$model $M_{w}^{m}$ is a
smooth warped product with base $B^{1} = [\,0,\, R[ \,\,\subset\,
\mathbb{R}$ (where $\, 0 < R \leq \infty$\,), fiber $F^{m-1} =
S^{m-1}_{1}$ (i.e. the unit $(m-1)-$sphere with standard metric),
and warping function $w:\, [\,0, \,R[\, \to \mathbb{R}_{+}\cup
\{0\}\,$ with $w(0) = 0$, $w'(0) = 1$, and $w(r) > 0\,$ for all
$\, r > 0\,$. The point $p_{w} = \pi^{-1}(0)$, where $\pi$ denotes
the projection onto $B^1$, is called the {\em{center point}} of
the model space. If $R = \infty$, then $p_{w}$ is a pole of
$M_{w}^{m}$.
\end{definition}

\begin{remark}\label{propSpaceForm}
The simply connected space forms $\kam$ of constant curvature $b$
can be constructed as  $w-$models with any given point as center
point using the warping functions
\begin{equation}
w(r) = Q_{b}(r) =\begin{cases} \frac{1}{\sqrt{b}}\sin(\sqrt{b}\, r) &\text{if $b>0$}\\
\phantom{\frac{1}{\sqrt{b}}} r &\text{if $b=0$}\\
\frac{1}{\sqrt{-b}}\sinh(\sqrt{-b}\,r) &\text{if $b<0$} \quad .
\end{cases}
\end{equation}
Note that for $b > 0$ the function $Q_{b}(r)$ admits a smooth
extension to  $r = \pi/\sqrt{b}$. For $\, b \leq 0\,$ any center
point is a pole.
\end{remark}

In the papers \cite{O'N,GreW,Gri,MP3,MP4}, we have a complete
description of these model spaces and their key properties.
In particular the sectional curvatures $K_{p_{w} , M_{w}}$ in the radial
directions from the center point $p_{w}$ are determined by the
radial function
\begin{equation}
K_{p_{w} , M_{w}}(\sigma_{x}) \, = \, K_{w}(r) \,
= \, -\frac{w''(r)}{w(r)} \quad ,
\end{equation}
and
 the mean curvature of the distance sphere of radius $r$ from the center point is
\begin{equation}\label{eqWarpMean}
\eta_{w}(r)  = \frac{w'(r)}{w(r)} = \frac{d}{dr}\ln(w(r))\quad .
\end{equation}

\subsection{The isoperimetric comparison spaces}

\label{secIsopCompSpace} Given the bounding functions $g(r)$,
$h(r)$ (when in the following no specific index is given, then $h$  represents  any one of  the bounding functions $h_{1}(r)$ or $h_{2}(r)$), and the ambient curvature controller function $w(r)$
described is subsections \ref{subsecurvature} and \ref{secModel},
as in \cite{MP5,HMP} we construct new model spaces $C^{\,m}_{w,
g, h}\,$. For completeness, we recall their construction:

\begin{definition}\label{stretchingfunct}
Given a smooth positive function $g(r) > 0$
satisfying $g(0)=1$ and $g(r)\leq 1\,\,{\textrm{for all \,}} x \in P$, a \emph{stretching function} $s$ is defined
as follows
\begin{equation}\label{eqstretching}
s(r) \, = \, \int_{0}^{r}\,\frac{1}{g(t)} \, dt \quad .
\end{equation}
It  has a well-defined inverse $r(s)$ for $s \in [\,0, s(R)\,]$
with derivative $r'(s) \, = \, g(r(s))$. In particular $r'(0)\, =
\, g(0) \, = \, 1$.
\end{definition}

\begin{definition}[\cite{MP5}] \label{defCspace}
The {\em{isoperimetric comparison space}} $C^{\,m}_{w, g, h}\,$ is defined as
the $W-$model space $M_{W}^{m}$ which has  base interval $B\,= \, [\,0, s(R)\,]$ and
warping function $W(s)$ defined by
\begin{equation}\label{defW}
W(s) \, = \, \Lambda^{\frac{1}{m-1}}(r(s)) \quad ,
\end{equation}
where the auxiliary function $\Lambda(r)$ satisfies the following differential
equation:
\begin{equation} \label{eqLambdaDiffeq}
\begin{aligned}
\frac{d}{dr}\,\{\Lambda(r)w(r)g(r)\} \, &= \, \Lambda(r)w(r)g(r)\left(\frac{m}{g^{2}(r)}\left(\eta_{w}(r) - h(r) \right)\right) \\
&= \, m\,\frac{\Lambda(r)}{g(r)}\left(w'(r) - h(r)w(r)
\right)\quad,
\end{aligned}
\end{equation}
and the following boundary condition:
\begin{equation} \label{eqTR}
\frac{d}{dr}_{|_{r=0}}\left(\Lambda^{\frac{1}{m-1}}(r)\right) = 1
\quad .
\end{equation}

\end {definition}

In spite of its relatively complicated
construction, $C^{\,m}_{w, g, h}\,$ is indeed a model space $M^m_W$ with a
well defined pole $p_{W}$ at $s = 0$: $W(s) \geq 0$ for all $s$
and $W(s)$ is only $0$ at $s=0$, where also, because of the
explicit construction in definition \ref{defCspace} and because of
 equation (\ref{eqTR}):  $W'(0)\, =
\, 1\,$.\\

Note that, when $g(r)=1 \,\,\,{\textrm{for all \,}} r$ and $h(r)=0\,\,\,{\textrm{for all \,\,}} r$,
then
the stretching function $s(r)=r$ and $W(s(r))=w(r) \,\,\,{\textrm{for all \,}} r$. In this case  $C_{w,g,h}^{m}$ simply reduces
to the $w$ warped model space $M^m_w$.\\

 The spaces $M_{W}^{m} = C_{w,g,h}^{m}$ will be applied as those spaces, where our bounds on the exit moment spectrum are attained.

\subsection{Balance conditions}

In the paper \cite{HMP} we considered and applied a balance condition on the
general model spaces $M^m_W$, that we shall also need in the sequel:

\begin{definition} \label{defBalCond}
The model space $M_{W}^{m} \, = \, C_{w, g, h}^{m}$ is
{\emph{$w-$balanced}} (respectively \emph{strictly $w-$ba\-lan\-ced}) if
the following holds for all  $s \in \, [\,0, s(R)\,]$:
\begin{equation}\label{eqBalA}
q_{W}(s)\left(\eta_{w}(r(s)) - h(r(s)) \right) \, \geq \, (>) \,g(r(s))/m \quad .
\end{equation}
Here $q_{W}(s)$ is the isoperimetric quotient function
\begin{equation} \label{eqIsopFunction}
\begin{aligned}
q_{W}(s) \, &= \, \frac{\Vol(B_{s}^{W})}{\Vol(S_{s}^{W})} \,
\\ &= \, \frac{\int_{0}^{s}\,W^{m-1}(t)\,dt}{W^{m-1}(s)}\,
\\ &= \,
\frac{\int_{0}^{r(s)}\,\frac{\Lambda(u)}{g(u)}\,du}{\Lambda(r(s))}
\quad .
\end{aligned}
\end{equation}

\end{definition}

\begin{remark} \label{remBalRef01}
 In particular the $w$-balance condition for $M_{W}^{m} \, = \, C_{w, g, h}^{m}$ implies that
\begin{equation} \label{eqEtaVSh}
\eta_{w}(r) \, - h(r) \, > \, 0 \quad 
\end{equation}
wherever $g(r) > 0$.
\end{remark}

\begin{remark} \label{remBalanceRef}
The above definition of a (strict) $w-$balance condition  for $M_{W}^{m}$ is clearly an extension of the balance
condition (from below) as defined in \cite[Definition 2.12]{MP4}. The condition in that paper
is obtained precisely when $g(r) \, = \, 1$ and $h(r) \, = \, 0$
for all $r \in [\,0, R]\,$ so that $r(s) \, =\, s$, $W(s)\, = \,
w(r)$, and
\begin{equation} \label{eqBalanceConstant}
q_{w}(r)\eta_{w}(r)\, \geq  \, 1/m \quad .
\end{equation}
This particular condition is of instrumental importance for the respective proofs of Theorem \ref{thmExFirst} and Theorem \ref{thm2.2}. For these settings it is easy to verify that
every warping function $w(r)$ which gives a negatively curved model space $M_{w}^{m}$ satisfies the strict version of (\ref{eqBalanceConstant}) for all $r$ -- using (\ref{eqIsopFunction}) for the functions $q_{w}(r)$,  see  also \cite[Observation 3.12 and Examples 3.13]{MP4}. In particular, the hyperbolic constant curvature spaces $M_{w}^{m} = \mathbb{H}^{m}(b)$, $b < 0$, all satisfy:
\begin{equation}
q_{w}(r)\eta_{w}(r)\, >  \, 1/m \quad .
\end{equation}
\end{remark}

\subsection{Comparison Constellations}

We now present the precise settings where our main results take
place, introducing the notion of {\em comparison constellations}
as they were previously defined in \cite{HMP}. For that purpose we
shall bound the previously introduced notions of radial curvature
and tangency by the corresponding quantities attained in the
special model spaces, the {\em isoperimetric comparison spaces}
defined above.

\begin{definition}\label{defConstellatNew1}
Let $N^{n}$ denote a complete Riemannian manifold with a
pole $p$ and distance function $r \, = \, r(x) \,
= \, \dist_{N}(p, x)$. Let $P^{m}$ denote an
unbounded complete and properly immersed submanifold in
$N^{n}$. Suppose $p \in P^m$ and suppose that the following  conditions are
satisfied for all $x \in P^{m}$ with $r(x) \in
[\,0, R]\,$:
\begin{enumerate}
\item The $p$-radial sectional curvatures of $N$ are bounded from below
by the $p_{w}$-radial sectional curvatures of
the $w-$model space $M_{w}^{m}$:
$$
\mathcal{K}(\sigma_{x}) \, \geq \,
-\frac{w''(r(x))}{w(r(x))} \quad .
$$

\item The $p$-radial mean curvature of $P$ is bounded from below by
a smooth radial function $h_{1}(r)$:
$$
\mathcal{C}(x)  \geq h_{1}(r(x)) \quad.
$$

\item The submanifold $P$ satisfies a {\it radial tangency
 condition} at $p\in P$, with smooth positive radial function $g(r)$ such that
\begin{equation}
\mathcal{T}(x) \, = \, \Vert \nabla^P r(x)\Vert
\geq g(r(x)) \, > \, 0  \quad {\textrm{for all}}
\quad x \in P \,\, .
\end{equation}
\end{enumerate}
Let $C_{w,g,h_{1}}^{m}$ denote the $W$-model with the
specific warping function $W: \pi(C_{w,g,h_{1}}^{m})
\to \mathbb{R}_{+}$  constructed in
Definition \ref{defCspace}, (subsection \ref{secIsopCompSpace}), via $w$, $g$, and $h = h_{1}$.
Then the triple $\{ N^{n}, P^{m}, C_{w,g,h_{1}}^{m}
\}$ is called an {\em{isoperimetric comparison
constellation bounded from below}} on the interval $[\,0, R]\,$.
\end{definition}

A \emph{constellation bounded from above} is given by the following dual setting
defining the special $W$-model spaces $C_{w,1,h_{2}}^{m}$ with the uniform choice $g=1$:

\begin{definition}\label{defConstellatNew2}
Let $N^{n}$ denote a Riemannian manifold with a
pole $p$ and distance function $r \, = \, r(x) \,
= \, \dist_{N}(p, x)$. Let $P^{m}$ denote an
unbounded complete and properly immersed submanifold in
$N^{n}$. Suppose the following  conditions are
satisfied for all $x \in P^{m}$ with $r(x) \in
[\,0, R]\,$:
\begin{enumerate}
\item The $p$-radial sectional curvatures of $N$ are bounded from above
by the $p_{w}$-radial sectional curvatures of
the $w-$model space $M_{w}^{m}$:
$$
\mathcal{K}(\sigma_{x}) \, \leq \,
-\frac{w''(r(x))}{w(r(x))} \quad .
$$

\item The $p$-radial mean curvature of $P$ is bounded from above by
a smooth radial function $h_{2}(r)$:
$$
\mathcal{C}(x)  \leq h_{2}(r(x)) \quad.
$$
\end{enumerate}

Let $C_{w,1,h_{2}}^{m}$ denote the $W$-model with the
specific warping function $W: \pi(C_{w,1,h_{2}}^{m})
\to \mathbb{R}_{+}$ constructed in
Definition \ref{defCspace} via $w$, $g=1$, and $h = h_{2}$.
Then the triple $\{ N^{n}, P^{m}, C_{w,1,h_{2}}^{m}
\}$ is called an {\em{isoperimetric comparison
constellation bounded from above}} on the interval $[\,0, R]\,$.
\end{definition}

\subsection{Laplacian Comparison}
\label{secLaplacecompar}

 We begin this section recalling the following
Laplacian comparison Theorem for manifolds with a pole (see
\cite{GreW, JK, Ma1, Ma2, MP3,MP4,MP5,MM} for more details and previous applications).
\begin{theorem} \label{corLapComp} Let $N^{n}$ be a manifold with a pole $p$, let $M_{w}^{m}$ denote a
$w-$model space with center $p_{w}$. Let us consider a smooth function $f: \erre_{+} \to \erre$ and the restricted distance function from the pole $r: P \to \erre$.

Then we have the following dual Laplacian inequalities for the  modified distance functions $$f\circ r: P \to \erre; \,\,f\circ r(x):= f(r(x))\,\,\forall x \in P$$

(i) Suppose that every $p$-radial sectional curvature at $x \in N
- \{p\}$ is bounded  by the $p_{w}$-radial sectional curvatures in
$M_{w}^{m}$ as follows:
\begin{equation}
\mathcal{K}(\sigma(x)) \, = \, K_{p, N}(\sigma_{x})
\geq-\frac{w''(r)}{w(r)}\quad .
\end{equation}

Then we have for every smooth function $f(r)$ with $f'(r) \leq
0\,\,\textrm{for all}\,\,\, r$, (respectively $f'(r) \geq
0\,\,\textrm{for all}\,\,\, r$):
\begin{equation} \label{eqLap1}
\begin{aligned}
\Delta^{P}(f \circ r) \, \geq (\leq) \, &\left(\, f''(r) -
f'(r)\eta_{w}(r) \, \right)
 \Vert \nabla^{P} r \Vert^{2} \\ &+ mf'(r) \left(\, \eta_{w}(r) +
\langle \, \nabla^{N}r, \, H_{P}  \, \rangle  \, \right)  \quad ,
\end{aligned}
\end{equation}
where $H_{P}$ denotes the mean curvature vector
of $P$ in $N$.\\

(ii) Suppose that every $p$-radial sectional curvature at $x
\in N - \{p\}$ is bounded  by the $p_{w}$-radial sectional
curvatures in $M_{w}^{m}$ as follows:
\begin{equation}
\mathcal{K}(\sigma(x)) \, = \, K_{p, N}(\sigma_{x})
\leq-\frac{w''(r)}{w(r)}\quad .
\end{equation}

Then we have for every smooth function $f(r)$ with $f'(r) \leq
0\,\,\textrm{for all}\,\,\, r$, (respectively $f'(r) \geq
0\,\,\textrm{for all}\,\,\, r$):
\begin{equation} \label{eqLap2}
\begin{aligned}
\Delta^{P}(f \circ r) \, \leq (\geq) \, &\left(\, f''(r) -
f'(r)\eta_{w}(r) \, \right)
 \Vert \nabla^{P} r \Vert^{2} \\ &+ mf'(r) \left(\, \eta_{w}(r) +
\langle \, \nabla^{N}r, \, H_{P}  \, \rangle  \, \right)  \quad ,
\end{aligned}
\end{equation}
where $H_{P}$ denotes the mean curvature vector of $P$ in $N$.
\end{theorem}

\section{Exit moment spectra of $R$-balls in model spaces}
\label{SpectrumBall}


We have the following result concerning the exit moment
spectrum of a geodesic $R$-ball $B^w_R \subseteq M^m_w$:

\begin{proposition}\label{propW1} Let $\tilde{u}_k$ be the solution of the  boundary value problems \eqref{eqmoments1},
defined on the geodesic $R$-ball $B^w_R$ in a warped model space
$M^m_w$.

Then
\begin{equation}\label{eq_u1}
\tilde{u}_1(r) \, = \, \int_{r}^{R}\,\frac{\int_0^t w^{m-1}(s)\,
ds}{w^{m-1}(t)}\,dt ,
\end{equation}
and
\begin{equation}\label{eq_uk}
\tilde{u}_k'(r) \, = -k\,\frac{\int_0^r w^{m-1}(s)
\tilde{u}_{k-1}(s)\, ds}{w^{m-1}(r)}.
\end{equation}
Therefore,
\begin{equation}\label{momentsW}
\widehat{\mathcal{A}}_k(B^w_R)=-\frac{1}{k+1}\, \tilde{u}_{k+1}'(R) \quad ,
\end{equation}
where $S^w_R$ is the geodesic $R$-sphere in $M^m_w$ .
\end{proposition}
\begin{proof}
Taking into account \eqref{eqWarpMean} and \eqref{eq_uk}, it is
easy to see that
\begin{equation}\label{laplautilde}
\Delta \tilde{u}_k=\tilde{u}_k''(r) + (m-1)\frac{w'(r)}{w(r)}
\tilde{u}_k'(r)=-k\,\tilde{u}_{k-1}(r).
\end{equation}
So, if
\begin{displaymath}
\tilde{u}_k(r)=k\,\int_r^R \frac{\int_0^t w^{m-1}(s)
\tilde{u}_{k-1}(s)\, ds}{w^{m-1}(t)}\,dt,
\end{displaymath}
the boundary condition $\tilde{u}_k(R)=0$ is satisfied and as a
consequence of the Maximum Principle for elliptic operators, the
functions $\tilde{u}_k$ are the only solutions to the boundary
value problems defined on $B^w_R$ and given by \eqref{eqmoments1}.

Therefore, applying the Divergence Theorem, we obtain

\begin{equation}\label{MomentDerivative}
\begin{aligned}
\widehat{\mathcal{A}}_k(B^w_R)\cdot\Vol(S^w_R)&=\int_{B^w_R} \tilde{u}_k\,
dV=-\frac{1}{k+1}\int_{B^w_R} \Delta \tilde{u}_{k+1}\,dV\\
&=-\frac{1}{k+1}\int_{S^w_R} \langle \nabla \tilde{u}_{k+1},
\nabla r \rangle\,dA=-\frac{1}{k+1}\, \tilde{u}_{k+1}'(R)\cdot\Vol(S^w_R) \quad ,
\end{aligned}
\end{equation}
and the claim is proved.
\end{proof}


\subsection{A key lemma}


Let us consider now an isoperimetric comparison model space  $M^m_W$  and let
$\tilde{u}^W_k$ be the radial functions given by \eqref{eq_uk},
which are the solutions of the problems \eqref{eqmoments1} defined
on the geodesic ball $B^W_{s(R)}$. We define the functions $f_k:
[\,0,R]\rightarrow \mathbb{R}$ as $f_k=\tilde{u}^W_k\circ s$, where
$s$ is the stretching function given by \eqref{eqstretching}.

Then we have the following lemma, which will be of instrumental importance for the proofs of the main results below:
\begin{lemma}\label{paren} 
Let $M^m_W$ be an  isoperimetric comparison model
space that is $w$-balanced  in the sense of Definition
\ref{defBalCond} with $h = h_{1}$ or $h = h_{2}$. Then for all $k\geq 1$,
\begin{displaymath}\label{eqParent}
f_k''(r)-f_k'(r)\eta_w(r)\geq 0 \quad.
\end{displaymath}
If $k \geq 2$ or  if $M^m_W$  is \emph{strictly} balanced, 
then the inequality is in fact a strict inequality:
\begin{displaymath} \label{eqParentSharp}
f_k''(r)-f_k'(r)\eta_w(r) > 0 \quad.
\end{displaymath}
\end{lemma}
\begin{proof}
By equation \eqref{eqstretching},
\begin{eqnarray}\label{fk''}
\nonumber
f''_k(r)&=&\tilde{u}^{W''}_k(s(r))(s'(r))^2+\tilde{u}^{W'}_k(s(r))
s''(r)\\
&=&\frac{1}{g^2(r)}(\tilde{u}^{W''}_k(s(r))-\tilde{u}^{W'}_k(s(r)) g'(r)).
\end{eqnarray}
Since the functions $\tilde{u}^W_k$ are the solution  of the
problems \eqref{eqmoments1} on $B^W_{s(R)}$, using equation
\eqref{laplautilde},
\begin{displaymath}
\tilde{u}^{W''}_k(s(r))=-k\,\tilde{u}^W_{k-1}(s(r))-(m-1)\frac{W'(s(r))}{W(s(r))}\,\tilde{u}^{W'}_{k}(s(r)).
\end{displaymath}
Taking into account the explicit construction of $M_{W}^{m}$, i.e. equations \eqref{defW} and
\eqref{eqLambdaDiffeq}, a straightforward computation shows that
\begin{displaymath}
(m-1)\frac{W'(s(r))}{W(s(r))}=\frac{m}{g(r)}(\eta_{w}(r)-h(r))-g(r)\eta_w(r)-g'(r),
\end{displaymath}
and consequently,
\begin{equation*}
\tilde{u}^{W''}_k(s(r))=-k\,\tilde{u}^W_{k-1}(s(r))-\frac{m}{g(r)}(\eta_{w}(r)-h(r))\,\tilde{u}^{W'}_{k}(s(r))+
(\eta_w(r)g(r)+g'(r))\,\tilde{u}^{W'}_{k}(s(r)).
\end{equation*}
 Replacing the expression of $\tilde{u}^{W"}_k(s(r))$ in equation
 \eqref{fk''} we obtain that
\begin{displaymath}
g^2(r)\,f_k''(r)=-k f_{k-1}(r)+(g^2(r)\eta_w(r)-m
(\eta_w(r)-h(r)))f_k'(r),
\end{displaymath}
and
\begin{equation}\label{parentesis}
g^2(r)(f_k''(r)-f_k'(r)\eta_w(r))=-k f_{k-1}(r)-m
(\eta_w(r)-h(r))f_k'(r).
\end{equation}
Since $f'_k(r)=\tilde{u}^{W'}_k(s(r))/g(r)< 0$, the functions $f_k$
are strictly decreasing in $]\,0,R\,]$ for all $k\geq 1$ and
consequently by \eqref{eq_uk}
\begin{eqnarray}
f'_k(r)&=& -k\,\frac{\int_0^{s(r)} W^{m-1}(s) \tilde{u}_{k-1}(s)\,
ds}{W^{m-1}(s(r))g(r)}=-k\,\frac{\int_0^r
\frac{\Lambda(t)}{g(t)}f_{k-1}(t)\, dt}{\Lambda(r)g(r)}\\
&\leq (<) & -k f_{k-1}(r)\frac{\int_0^r \frac{\Lambda(t)}{g(t)}\,
dt}{\Lambda(r)g(r)}=-k f_{k-1}(r)q_W(s(r))/g(r),\label{eqkineq}
\end{eqnarray}
where the last equality is obtained using equation
\eqref{eqIsopFunction}. Note that we can assume that
$\tilde{u}_0\equiv 1$ and therefore $f_0\equiv 1$ too, so that only in the
case $k=1$ can we have equality in (\ref{eqkineq}).

Finally, combining the above inequality with equation
\eqref{parentesis} we get:
\begin{displaymath}
g^3(r)(f_k''(r)-f_k'(r)\eta_w(r))\, \geq (>) \, k f_{k-1}(r)\left(-g(r)+m\,
q_W(s(r))(\eta_w(r)-h(r))\right)\geq (>) \, 0 
\end{displaymath}
by the balance condition \eqref{eqBalA} -- respectively the strict balance condition --  and the fact that $g$ and
$f_{k-1}$ are positive functions.
\end{proof}


\section{Lower and Upper bounds for the isoperimetric exit moments}\label{MainSect}


We are now ready to prove the first of our main results.

\begin{theorem} \label{thm2.1}
Let $\{ N^{n}, P^{m}, C_{w,g,h_{1}}^{m} \}$ denote a comparison
constellation boun\-ded from below in the sense of Definition
\ref{defConstellatNew1}. Assume that $M^m_W = C_{w,g,h}^{m}$ is
$w$-balanced in the sense of Definition \ref{defBalCond}. Let $D_R$ be an
extrinsic $R$-ball in $P^m$,
 with center at a point $p \in P$ which also serves as a pole in $N$.  According to remark \ref{remTrivTop}, our assumption $g(r(x)) > 0$ implies trivial topology of the extrinsic ball $D_R$. For all $k\geq 0$, i.e. for the extended exit moment spectrum, we also have:
\begin{equation} \label{eqMain}
\widehat{\mathcal{A}}_k(D_R) \geq
\widehat{\mathcal{A}}_k(B^{W}_{s(R)})\quad ,
\end{equation}
where $B^{W}_{s(R)}$ is the geodesic $s(R)$-ball in $C_{w,g,h_{1}}^{m}$.
\end{theorem}

\begin{proof}

Consider the functions $f_k=\tilde{u}^W_k\circ s$ of Lemma
\ref{paren}. Let $r$ denote the smooth distance to the pole $p$ on
$M$. We define $v_k:D_R\rightarrow \mathbb{R}$ by
$v_k(q)=f_k(r(q))$.

Using Theorem \ref{corLapComp}, Lemma \ref{paren}, equation
\eqref{parentesis} and the fact that $f'_k(r)\leq 0$, we have that
\begin{eqnarray}\label{laplaineqv_k}
\Delta^P v_k = \Delta^P(f_k\circ r)&\geq&
(f_k''(r)-f_k'(r)\eta_w(r))\|\nabla^P r\|^2+
m\,f_k'(r)(\eta_w(r)-h_{1}(r))\\
&\geq&
(f_k''(r)-f_k'(r)\eta_w(r))\cdot g^{2}(r)+
m\,f_k'(r)(\eta_w(r)-h_{1}(r))\\
\nonumber &=&-k f_{k-1}(r)=-k\, v_{k-1}, \quad \text{on}\,\,D_R
\end{eqnarray}

 Now, we are going to prove {\emph{inductively}} that if we denote by
 $u_k$ the solutions of the hierarchy of boundary value problems on
 $D_R$ given by \eqref{eqmoments1}, then $v_k\leq u_k$ on $D_R$.

For $k=1$, since $f_0$ is assumed to be identically $1$,
inequality \eqref{laplaineqv_k} gives us that
\begin{displaymath}
\Delta^P v_1 \geq -1=\Delta^P u_1,
\end{displaymath}
so $\Delta^P (v_1 - u_1)\geq 0$ on $D_R$ and $(v_1-u_1)=0$ on
$\partial D_R$. Applying the Maximum Principle we conclude that
$v_1\leq u_1$ on $D_R$.

Suppose now that $v_k \leq u_k$ on $D_R$, then as a consequence of
inequality \eqref{laplaineqv_k} we get
\begin{displaymath}
\Delta^P v_{k+1}\geq -(k+1)\,v_k \geq -(k+1)\,u_k=\Delta^P
u_{k+1},
\end{displaymath}
and $(v_{k+1}-u_{k+1})=0$ on $\partial D_R$, so applying again the
Maximum Principle we have $v_{k+1}\leq u_{k+1}$.

Summarizing we have so far:  $v_k\leq u_k$ and $\Delta^P v_k \geq \Delta^P u_k$ on
$D_R$ for all $k\geq 1$. Taking  these inequalities into account and applying Divergence theorem we
then get

\begin{eqnarray*}
\widehat{\mathcal{A}}_k(D_R)\cdot\Vol(\partial D_{R})&=&\int_{D_R} u_k d\,V =-\frac{1}{k+1}\int_{D_R}
\Delta^P u_{k+1} d\,V \\
&\geq& -\frac{1}{k+1}\int_{D_R} \Delta^P v_{k+1}
d\,V=-\frac{1}{k+1}\int_{\partial D_R} \langle\nabla^P v_{k+1},
\frac{\nabla^P r}{\|\nabla^P r\|}\rangle
d\,A\\
&=&-\frac{1}{k+1}f_{k+1}'(R)\int_{\partial D_R} \|\nabla^P r\|
d\,A.
\end{eqnarray*}
Since $f_{k+1}'(R)=\tilde{u}^{W'}_{k+1}(s(R))/g(R) \leq 0$ and
$\|\nabla^P r\| \geq g(r)$, we conclude that
\begin{eqnarray*}
\widehat{\mathcal{A}}_k(D_R) \geq
-\frac{1}{k+1}\,\frac{\tilde{u}^{W'}_{k+1}(s(R))}{g(R)}\,g(R)\,
=\widehat{\mathcal{A}}_k(B^W_{s(R)}),
\end{eqnarray*}
by equation \eqref{momentsW}. And this proves the claim in (\ref{eqMain}).
\end{proof}

\begin{theorem} \label{thm2.2}
Let $\{ N^{n}, P^{m}, C_{w,1,h_{2}}^{m} \}$  denote a comparison
constellation boun\-ded from above. Assume that
$M^m_W=C_{w,1,h_{2}}^{m}$ is $w$-balanced in the sense of Definition \ref{defBalCond}. Let $D_R$ be a
smooth precompact extrinsic $R$-ball in $P^m$
 with center at a point $p \in P$ which also serves as a pole in $N$.  Then, for all $k\geq 0$, i.e. for the extended isoperimetric exit moment spectrum we have:
\begin{equation} \label{eqMain3}
\widehat{\mathcal{A}}_k(D_R) \leq
\widehat{\mathcal{A}}_k(B^{W}_{R}) \quad ,
\end{equation}
where $B^{W}_{R}$ is the geodesic ball in $C_{w,1,h_{2}}^{m}$.

If $M_{W}^{m}$ is strictly balanced then equality in (\ref{eqMain3}) for some fixed radius $R$ and some fixed $k \geq 0$ implies that
$D_R$ is a geodesic cone in $N$ and that the equality is in fact attained for all $k \geq 0$ and for every smaller $p$-centered extrinsic ball in $P^{m}$.
\end{theorem}

\begin{proof}
The proof of this theorem follows closely the lines of the proof of Theorem
\ref{thm2.1}. Since there are, however, some crucial and obvious differences we take this space to point them out explicitly.
In the present case we have  $s(r)=r$ because $g(r) \equiv 1$ (see equation \eqref{eqstretching}). Therefore
$f_{k+1}=\tilde{u}^W_{k+1}$ so that $v_{k+1}= \tilde{u}^W_{k+1}\circ
r$. Thence $v_{k+1}$ is the solution of the boundary
value problems \eqref{eqmoments1} on $B^W_R$  transplanted to
$D_R$.

The new geometric setting given by the {\em comparison
constellation bounded from above} gives us now:
\begin{eqnarray}
\Delta^P v_k = \Delta^P(f_k\circ r)&\leq&
(f_k''(r)-f_k'(r)\eta_w(r))\|\nabla^P r\|^2+
m\,f_k'(r)(\eta_w(r)-h_{2}(r))   \label{eqAbov4} \\
&\leq&
(f_k''(r)-f_k'(r)\eta_w(r)) +
m\,f_k'(r)(\eta_w(r)-h_{2}(r))   \label{eqAbov3}  \\
\nonumber &=&-k f_{k-1}(r)=-k\, v_{k-1}, \quad \text{on}\,\,D_R \quad.
\end{eqnarray}
Again we prove {\emph{inductively}} that if
 $u_k$ denotes the family of solutions of the hierarchy of boundary value problems on
 $D_R$ given by \eqref{eqmoments1}, then $v_k\geq u_k$ on $D_R$.

For $k=1$, since $f_0$ is still assumed to be identically $1$,
inequalities \eqref{eqAbov3} and \eqref{eqAbov4} give us that
\begin{displaymath}
\Delta^P v_1 \leq -1=\Delta^P u_1,
\end{displaymath}
so $\Delta^P (v_1 - u_1)\leq 0$ on $D_R$ and $(v_1-u_1)=0$ on
$\partial D_R$. Applying the Maximum Principle we conclude that
$v_1\geq u_1$ on $D_R$.

Suppose now that $v_k \geq u_k$ on $D_R$, then again as a consequence of
inequalities \eqref{eqAbov4} and \eqref{eqAbov3} we get
\begin{displaymath}
\Delta^P v_{k+1}\leq -(k+1)\,v_k \leq -(k+1)\,u_k=\Delta^P
u_{k+1},
\end{displaymath}
and $(v_{k+1}-u_{k+1})=0$ on $\partial D_R$, so applying again the
Maximum Principle we have $v_{k+1}\geq u_{k+1}$.

We have:  $v_k\geq u_k$ and $\Delta^P v_k \leq \Delta^P u_k$ on
$D_R$ for all $k\geq 1$. The Divergence theorem gives the claim in (\ref{eqMain3}):
\begin{eqnarray}
\nonumber\widehat{\mathcal{A}}_k(D_R)\cdot\Vol(\partial D_{R})&=&\int_{D_R} u_k d\,V =-\frac{1}{k+1}\int_{D_R}
\Delta^P u_{k+1} d\,V \\
&\leq& -\frac{1}{k+1}\int_{D_R} \Delta^P v_{k+1} d\,V \label{eqAbov01}\\
\nonumber&=&-\frac{1}{k+1}f_{k+1}'(R)\int_{\partial D_R} \|\nabla^P r\|
d\,A\\
&\leq & \widehat{\mathcal{A}}_k(B^W_{R})\cdot\Vol(\partial D_{R})\quad . \label{eqAbov02}
\end{eqnarray}
Suppose that $M_{W}^{m}$ is strictly balanced and that we have equality in (\ref{eqMain3}). Then we must have equalities in \eqref{eqAbov02}, \eqref{eqAbov01}, and \eqref{eqAbov3} as well.
In particular the last mentioned equality gives $\Vert \nabla^P r\Vert \equiv 1$ because we have from \eqref{eqParentSharp} that $(f_k''(r)-f_k'(r)\eta_w(r)) > 0$. Therefore $\nabla^{P} r = \nabla^{N} r$ and $D_{R}$ is a geodesic cone swept out by the radial geodesics from $p$. 
\end{proof}

\section{Intrinsic and constant curvature results} \label{secIntrinsic}
In this short section we finally
show how to obtain the results stated in the introduction from Theorem \ref{thm2.1} and Theorem \ref{thm2.2}.

\begin{proof}[Proof of Theorem \ref{thmExFirst}]
This theorem follows immediately from Theorem \ref{thm2.2} once we show that the comparison space  $M_{W}^{m}$ is strictly $w$-balanced. But  we have $g=1$ and $h_{2}=0$ so that $M_{W}^{m}$ is  $M_{w}^{m} = \mathbb{H}^{m}(b)$, $b < 0$, which is strictly $w$-balanced according  to remark \ref{remBalanceRef}. The equality case  gives even more significant rigidity: Since $D_{R}$ is here a \emph{minimal} geodesic cone, then  by analytic continuation $D_{R}$ and in fact all of $P^{m}$ is totally geodesic in the hyperbolic space $\mathbb{H}^{n}(b)$, see \cite{Ma1}.
\end{proof}

\begin{proof}[Proof of Theorem \ref{thm2.1intrinsic}]
We consider the intrinsic versions of (the proofs of) Theorem \ref{thm2.1} and Theorem \ref{thm2.2}assuming that $P^m=N^n$. In this case, the extrinsic
distance to the pole $p$ becomes the intrinsic distance in $N$,
so, the extrinsic domains $D_R$ become the geodesic
balls $B^N_R$ of the ambient manifold $N$ and for all $x \in P$ we have:
\begin{eqnarray*}
\nabla^P r(x)&=&\nabla r (x),\\
H_P(x)&=&0.
\end{eqnarray*}
As a consequence, $\Vert\nabla^P r\Vert \equiv 1$, so $g(r(x))=1$ and
$\mathcal{C}(x)=h_{1}(r(x))=h_{2}(r(x)) = 0$. The stretching function becomes the
identity $s(r)=r$, $W(s(r))=w(r)$, and the isoperimetric
comparison spaces $C_{w,g,h_{1}}^m$ and  $C_{w,1,h_{2}}^m$reduce to the same auxiliary model
space $M^m_w$. Since $\|\nabla r\| \equiv 1$, we do not need to control
the sign of $(f_k''(r)-f_k'(r)\eta_w(r))$ in equations
\eqref{eqLap1} and \eqref{eqLap2}. For this reason it is not
necessary to assume any $w$-balance conditions in these cases.
The theorem and the two-sided bounds in (\ref{eqMain6}) then follow directly from the inequalities in Theorem \ref{thm2.1} and Theorem \ref{thm2.2}.
If equality is satisfied, then $B^N_R$ has all its radial curvatures equal to the radial curvatures of $M^m_w$, hence they are isometric, see \cite{MP4}.
\end{proof}

\end{document}